\documentclass[12pt]{amsart}

\usepackage{latexsym, amsmath, amscd, amssymb, amsthm,longtable}

\usepackage[english]{babel}

\newtheorem{te}{Theorem}

\begin{document}

\vspace{3cm}

\title {  The star sequence   and the general first Zagreb index}

\author{Leonid Bedratyuk, Oleg Savenko}

\address{Khmelnitskiy national university, Instytutska, 11,  Khmelnitsky, 29016, Ukraine}

\email{leonid.uk@gmal.com}

\begin{abstract}{For a simple graph we introduce a notion of the star sequence and prove that the star sequence and the frequently sequences of a graph  are inverses of each other   from a combinatorial point of view. As a consequence, we express  the general  first Zagreb index in terms of the star sequence. Also,   we calculate the ordinary generating function  and find a linear recurrence relation for the sequence of  the general  first Zagreb indexes.

}

\end{abstract}

\maketitle

\section{Introduction}

Let $G$ be a simple graph whose vertex and edge sets are $V (G)$ and $E(G)$, respectively. Let $\deg(v)$ be the degree  of the vertex $v \in  V (G).$ For any real $p$ the general Zagreb index is defined by  

$$
Z_p(G)=\sum_{v \in V(G)} \deg(v)^p,
$$

see \cite{IG} for more details.

Put $n=|V(G)|$ and $m=|E(V)|$. We have $Z_0(G)=n$ and $Z_1(G)=2m$ due to the  Handshaking lemma.

 The well-known ordinary first Zagreb index $M_ 1$  is a special case of the general Zagreb index $Z_p(G)$ for  $p=2.$ 

It is easy to see  ( \cite[ Theorem 8.1]{FH})  that 

$$
Z_2(G)=2m+2q,
$$

where  and $q$ is the number of subraphs that are isomorphic to the path graph $P_3.$ Taking into account that $m$ is the number of subraphs that are isomorphic to the star graph $S_1$ and that $P_3 \cong S_2$   we may rewrite the expression for $Z_2(G)$  as follows

$$
Z_2(G)=2 S_1(G)+2 S_2(G), 
$$

where $S_1(G), S_2(G)$ denote the number of  subgraph of $G$  that are isomorphic to the star graphs $S_1$  and $S_2$, respectivelly.

In the paper, we  generalize this  expression for the general first Zagreb index $Z_p(G)$  for any natural  $p.$

Let  $S_k$ be the star graph   defined as the complete bipartite graph $K_{1,k}$. 

Denote by $S_k(G)$ the number of subgraphs  of $G$  that are isomorphic to the star  $S_k.$ 

 For instance $S_1(G)$ is equal to the number  of edges of  $G$  and $S_{n-1}(G)$  is equal to the number of vertices of maximal degrees  $n-1,$ $n=|V(G)|$.  

The sequence 

$$
2S_1(G), S_2(G), \ldots, S_{n-1}(G)
$$

is called  a \textit{star sequence} of a graph $G.$

We can define some types of  graphs in terms of its sequence star.  For example, 

for  a graph $G$  to be  the path  $P_n$ it is necessary and sufficient to have    $S_1(G)=n-1,  S_2(G)=n-2,$ and $ S_i(G)=0$  for $i>2.$ Similarly, a graph $G$ is  a $k$-regular iff the following conditions hold:

$$2S_1=\binom{k}{1}S_k(G), S_i=\binom{k}{i}S_k(G), \text{  and $ S_i(G)=0$ for  $i>k.$} $$

The main result of the  paper is the formula for expressing  of the general first Zagreb index in terms of the star sequence

$$
Z_p(G)=2S_1(G)+\sum_{i=2}^{p} i! {\left\{ {p\atop i}\right\}} S_i(G),
$$

here  $\displaystyle \left\{ {p \atop i}\right\}$ are the Stirling numbers of the second kind.

The result  can be considered  as a wide generalization of the  Handshaking Lemma.

Also, we calculate the generating function for the integer sequence $\{Z_p(G) \}:$

$$
\sum_{p=0}^\infty Z_p(G) t^p =\frac{\displaystyle \sum_{k=0}^{n-1} \left( \sum_{i=0}^k \left [ {n+1 \atop n+1-(k-i)}\right ] Z_i(G) \right) t^k }{(1-t)(1-2 t) \cdots (1-n t)},
$$

and find the recurrence relation for the integer sequence $\{ Z_p(G) \}:$

$$
Z_p(G)+\sum_{i=1}^{n} \left [ {p+1 \atop p+1- i}\right ] Z_{p-i}(G)=0, p \geq n, 
$$

here  $\displaystyle \left[ {p \atop i}\right]$ are the signed Stirling numbers of the first  kind.

\section{Star and frequently sequences}

Let $f_i$ denotes the number of vertices of degree $i, i=0\ldots n-1.$

The integer sequence  

$$
f_1, f_2, \ldots, f_{n-1},
$$

is called \textit{the frequency  sequence} of a graph. 

 The frequency  sequence  were studied intensively by several  authors,  see \cite{CHI}, \cite{MR}.

There exists a close connection between the star sequence  and the frequency  sequence of a graph.

Let us recall that two sequences $\{a_n \}, \{ b_n \}$ that satisfies the following conditions 

$$
a_i=\sum_{k=i}^n \binom{k}{i} b_k,  b_i=\sum_{k=i}^n (-1)^{k-i}\binom{k}{i} a_k, 
$$

are an example of a pair of  inverse sequences, see  \cite{Rio}  for more details.

The following theorem holds.

\begin{te} Let $G$ be a simple graph. Then its star and frequency sequences are inverses of each other: 

\begin{gather*}
f_i=\sum_{k=i}^{n-1} (-1)^{k-i} \binom{k}{i}S_k(G),  1 < i \leq n-1,\\
f_1=2S_1(G)+\sum_{k=2}^{n-1} (-1)^{k-i} k S_k(G),
\end{gather*}

and 

\begin{gather*}
2S_1(G)=\sum_{i=1}^{n-1} i f_i,\hspace{1cm}
S_k(G)=\sum_{i=k}^{n-1} \binom{i}{k} f_i, 1 < k \leq n-1.
\end{gather*}

\end{te}

\begin{proof}

Let us count the number of vertices that have the degree $k>1.$ It is clear that the number of vertices of the maximal degree    $n-1$  is equal to the number of $G$-subgraphs that are isomorphic to the star graph  $S_{n-1}$. Thus  $f_{n-1}=S_{n-1}(G).$ To count the number of vertices of degree  $n-2$, observe that the number  $f_{n-2}$ is not equal to  $S_{n-2}(G)$. In fact, the star graph $S_{n-2}$ is subgraph  of  $S_{n-1}$ and  some of the vertices of degree $n-2$ will be counted twice. Since each graph  $S_{n-1}$ consists of exactly  $\binom{n-1}{n-2}$ subgraphs  $S_{n-2}$ then 

$$
f_{n-2}=S_{n-2}(G)-\binom{n-1}{n-2} S_{n-1}(G).
$$

 It is becoming obvious that for arbitrary  $f_k, k>1$ we can use  the inclusion-exclusion principle:

$$
f_k=S_{k}(G)-\binom{k+1}{k}S_{k+1}(G)+\binom{k+2}{k}S_{k+2}(G)+\cdots+(-1)^{n-k} \binom{n-1}{k}S_{n-1}(G).
$$

For the case $k=1$ there exist  $S_1(G)$ edges and every one  of them has  $2$  vertices of degree $1$. Thus 

$$
f_1=2S_1(G)-2S_2(G)+3S_3(G)-4S_4(G)+\cdots+(-1)^{n-2}S_{n-1}(G).
$$

 Now we are able to express the star sequence of graph $G$  in terms of  its  frequency sequence.

\begin{align*}
&2S_1(G)=\sum_{i=1}^{n-1} \binom{i}{1} f_i,  \text{ the Handshaking lemma } \\
&S_k(G)=\sum_{i=k}^{n-1} \binom{i}{k} f_i.
\end{align*}

\end{proof}

The following theorem can be proved  in the same way as Theorem 1.

\begin{te}  We have

\begin{gather*}
2S_1(G)+\sum_{i=2}^{n-1} (-1)^{i-1} S_i(G)=\sum_{i=1}^{n-1} f_i,\\
2S_1(G)+\sum_{i=2}^{n-1} (-1)^{i-1} i^m  S_i(G)=\sum_{k=1}^{m } (-1)^k  k! \left\{ {m \atop k}\right\} f_k.
\end{gather*}

\end{te}

For the case $m=0$  we get 

\begin{gather*}
\sum_{i=1}^{n-1} f_i=2S_1(G)+\sum_{i=2}^{n-1} (-1)^{i-1}S_i(G),\\
\end{gather*}

This fact implies the following interesting result:

\begin{te}  

$$
\sum_{uv \in G(V)}\left( \frac{1}{d_u}+\frac{1}{d_v} \right)=S_1(G)+\sum_{i=1}^{n-1} (-1)^{i-1} S_i(G).
$$

\end{te}

\begin{proof}

From \cite{DM}  we know that 

$$
\sum_{uv \in G(V)}\left( \frac{1}{d_u}+\frac{1}{d_v} \right)=n-f_0.
$$

The number of isolated vertices  $f_0$  of a graph can be determined from its star sequence. Indeed

$$
2S_1(G)+\sum_{i=2}^{n-1} (-1)^{i-1} S_i(G)=\sum_{i=1}^{n-1}f_i=\sum_{i=0}^{n-1}f_i-f_0=n-f_0.
$$

Thus 

$$
\sum_{uv \in G(V)}\left( \frac{1}{d_u}+\frac{1}{d_v} \right)=S_1(G)+\sum_{i=1}^{n-1} (-1)^{i-1} S_i(G).
$$ 

\end{proof}

\section{The first general Zagreb index}

Now we can express the first general Zagreb index in term of star sequence.

\begin{te}

$$
Z_p(G)=2S_1(G)+\sum_{i=2}^{p} i! {\left\{ {p\atop i}\right\}} S_i(G),
$$

here  $\displaystyle \left\{ {p \atop i}\right\}$  is  the Stirling number of the second kind. 

\end{te}

\begin{proof}

 It is easy to see that 

$$
Z_p(G)=\sum_{i=1}^{n-1}i^p f_i.
$$

Then

\begin{gather*}
\sum_{i=1}^{n-1} i^p f_i=S_1(G)+\sum_{i=1}^{n-1} i^p \sum_{k=i}^{n-1} (-1)^{k-i} \binom{k}{i}S_k(G)=\\=S_1(G)+\sum_{k=1}^{n-1 } \left( \sum_{i=1}^k i^p (-1)^{k-i} \binom{k}{i} \right) S_k(G).
\end{gather*}

By using the well-known identity  (see  \cite{GKP}, the identity (6.19))

$$
\sum_{i=1}^k i^p (-1)^{k-i} \binom{k}{i}=k! \left\{ {p \atop k}\right\},
$$

we get

$$
\sum_{i=1}^{n-1} i^p f_i=S_1(G)+\sum_{k=1}^{n-1 } k! \left\{ {p \atop k}\right\} S_k(G).
$$

Since  $\displaystyle \left\{ {p \atop k}\right\} = 0 $ for $k>p$ we  thus obtain

$$
\sum_{i=1}^{n-1} i^p f_i=S_1(G)+\sum_{k=1}^{p } k! \left\{ {p \atop k}\right\} S_k(G).
$$

\end{proof}

Let 

$$
\mathcal{G}(Z,t)=\sum_{p=0}^\infty Z_p(G) t^p,
$$

be the ordinary generating functions of the sequence of the first  general  Zagreb indexes.

Let us express the generating functions $\mathcal{G}(Z,t)$  in terms of the   frequently and star sequences. The following theorem holds.

\begin{te} Let $\mathcal{G}(Z,t)$   be the  ordinary generating functions of the sequence of the first  general  Zagreb indexes. Then  

\begin{align*}
(i) & \,\mathcal{G}(Z,t)=\frac{\displaystyle \sum_{k=0}^{n-1} \left( \sum_{i=0}^k \left [ {n+1 \atop n+1-(k-i)}\right ] Z_i(G) \right) t^k }{(1-t)(1-2 t) \cdots (1-n t)},\\
(ii) & \, Z_p(G)+\sum_{i=1}^{n} \left [ {p+1 \atop p+1- i}\right ] Z_{p-i}(G)=0, p>n, 
\end{align*}

 where $\displaystyle \left [ {p \atop i}\right ]$  is the signed Stirling numbers of the first kind.

\end{te}

\begin{proof}

$(i)$    Since

$$
Z_p(G)=2S_1(G)+\sum_{i=2}^{p} i! {\left\{ {p\atop i}\right\}} S_i(G),
$$

then 

\begin{gather*}
\mathcal{G}(Z,t)=\sum_{p=0}^\infty Z_p(G) t^p=Z_0(G)+\sum_{p=1}^\infty \left(2S_1(G)+\sum_{i=2}^{p} i! {\left\{ {p\atop i}\right\}} S_i(G) \right) t^p=\\=Z_0(G)+\sum_{p=1}^\infty 2S_1(G) t^p+\sum_{p=1}^\infty \sum_{i=2}^{p} i! {\left\{ {p\atop i}\right\}} S_i(G) t^p=\\
=Z_0(G)+2S_1(G)\sum_{p=1}^\infty  t^p+ \sum_{i=2}^{\infty} \left(\sum_{p=1}^\infty  {\left\{ {p\atop i}\right\}}  t^p \right)i!S_i(G)=\\=
n+\frac{2S_1(G) t}{1-t}+\sum_{i=2}^{n} \frac{i!S_i(G) t^i}{(1-t)(1-2 t) \cdots (1-i t)}.
\end{gather*}

Here we used the well-known formula for the generating function for the Stirling numbers of the second kind:

$$
\sum_{p=1}^{\infty}  {\left\{ {p\atop i}\right\}}  t^p=\frac{ t^i}{(1-t)(1-2 t) \cdots (1-i t)}.
$$

After simplification we get 

$$
\mathcal{G}(Z,t)=\frac{a_0+a_1 t+a_2 t^2+\cdots+a_{n-1} t^{n-1}}{(1-t)(1-2 t) \cdots (1-n t)}, 
$$

for some unknown numbers $a_0, a_1, \ldots, a_{n-1}.$  To define the numbers let us observe that 

$$
(1-t)(1-2 t) \cdots (1-n t)=\sum_{i=0}^n  \left [ {n+1 \atop n+1-i}\right ] t^i.
$$

Now 

\begin{gather*}
a_0+a_1 t+a_2 t^2+\cdots+a_{n-1} t^{n-1}=\\=(1-t)(1-2 t) \cdots (1-n t) \left(  Z_0(G)+Z_1(G) t+\cdots+Z_n(G) t^n+\cdots \, \right)=\\
=\left(\sum_{i=0}^n  \left [ {n+1 \atop n+1-i}\right ] t^i \right) \left(  Z_0(G)+Z_1(G) t+\cdots+Z_n(G) t^n+\cdots \, \right).
\end{gather*}

Equating coefficients of $t^k$  yelds 

$$
a_k=\sum_{i=0}^k  \left [ {n+1 \atop n+1-(k-i)}\right ] Z_i(G).
$$

Therefore

\begin{gather*}
\mathcal{G}(Z,t)=
\frac{\displaystyle \sum_{k=0}^{n-1} \left( \sum_{i=0}^k \left [ {n+1 \atop n+1-(k-i)}\right ] Z_i(G) \right) t^k }{(1-t)(1-2 t) \cdots (1-n t)}.
\end{gather*}

$(ii)$ 

The generating function $\mathcal{G}(Z,t)$ is a rational function  with the denominator 

$$
(1-t)(1-2 t) \cdots (1-n t)=\sum_{i=0}^n  \left [ {n+1 \atop n+1- i}\right ] t^i.
$$

Then Theorem 4.1.1 \cite{St-En_1} immediatelly implies that  the sequence  of the first general Zagreb indices  $Z_0(G), Z_1(G), \ldots Z_p(G), \ldots $ satisfies the  recurrence relation

$$
Z_p(G)+\sum_{i=1}^{n} \left [ {n+1 \atop n+1- i}\right ] Z_{p-i}(G)=0,
$$

for all  $p  \geq n.$

\end{proof}


\begin{thebibliography}{30}

\bibitem{IG}

I. Gutman, An Exceptional Property of First Zagreb Index, \textit{MATCH Commun. Math. Comput. Chem}. 72 (2014) 733-740



\bibitem{FH}

F. Harary. \textit{Graph Theory.} Addison-Wesley, Reading, MA, 1969. 



\bibitem{CHI}

 P. Z. Chinn, The frequency partition of a graph, \textit{“Recent Trends in Graph Theory”}

(M. Copabianco, ed.), pp. 69-70, Springer-Verlag, 1971. 



\bibitem{MR}

T. Manadeva Rao, Frequency Sequences in Graphs, \textit{Journal of Combinatorial Theory}, 17, 19-21 (19





\bibitem{Rio}

J. Riordan,  \textit{Combinatorial Identities.} New York: Wiley,  1979.



\bibitem{GKP}

R. L. Graham, D. E. Knuth and O. Patashnik. \textit{Concrete Mathematics}. Addison-Wesley, Reading, 1989



\bibitem{DM}

T. Doslic, B. Furtula, A. Graovac, I. Gutman, S. Moradi, Z. Yarahmadi: On Vertex-Degree-Based Molecular Structure Descriptors,

\textit{MATCH Commun. Math. Comput. Chem.} 66 (2011) 613-626.





\bibitem{St-En_1} R. Stanley, 

\textit{Enumerative combinatorics.} Vol. 1.,

Cambridge Studies in Advanced Mathematics. 49. Cambridge: Cambridge University Press, 1997.

\end{thebibliography}
\end{document}